\documentclass{article}
\usepackage{amsmath}
\usepackage{amssymb}
\usepackage{amsthm}
\usepackage{euscript}
\usepackage{graphicx}
\usepackage{epstopdf}

\title{The \Apery\ Numbers As a\\
Stieltjes Moment Sequence}
\author{G. A. Edgar}
\date{January 20, 2017\\Appendix, August 15, 2020}

\theoremstyle{plain}
\newtheorem{pr}{Proposition}
\newtheorem{thm}[pr]{Theorem}
\newtheorem{cor}[pr]{Corollary}
\newtheorem{lem}[pr]{Lemma}

\theoremstyle{remark}

\newtheorem{de}[pr]{Definition}
\newtheorem{no}[pr]{Notation}

\allowdisplaybreaks

\newcommand{\Def}[1]{\textbf{\itshape #1}} 

\renewcommand{\phi}{\varphi}
\renewcommand{\epsilon}{\varepsilon}

\newcommand{\Hn}[7]{\mathrm{Hn}\left(\begin{array}{c | r l | }%
#1 & #3\; , & #4 \\ #2 & #5\; ; & #6 \end{array} \;#7\right)}

\newcommand{\co}{c_0}

\newcommand{\carat}{$\wedge$}
\newcommand{\Apery}{Ap\'ery}	
\newcommand{\sr}{square root}
\renewcommand{\th}{\boldsymbol\theta}

\newcommand{\FF}{\mathbf{K}}

\begin{document}
\maketitle
\setcounter{tocdepth}{2}

The \Def{\Apery\ sequence} \cite{apery}~\cite[\texttt{A005259}]{OEIS} is
\begin{equation*}
	A_n = \sum_{k=0}^n \binom{n}{k}^2 \binom{n+k}{k}^2 ,
	\qquad n=0,1,2,\cdots .
\end{equation*}
From the reference (or a CAS\footnote{
\texttt{SumTools[Hypergeometric][Zeilberger](binomial(n,k)\carat2*binomial(n+k,k)\carat2,n,k,En);}})
we find that it satisfies the recurrence
\begin{equation}\label{rec}
	(n+1)^3A_{n+1}
	-(34n^3+51n^2+27n+5)A_n
	+n^3A_{n-1} = 0 ,
\end{equation}\begin{equation*}
	A_0 = 1,\qquad A_1 = 5 .
\end{equation*}
We will show that the sequence $(A_n)$ is a \Def{Stieltjes 
moment sequence}.  In fact:

\begin{thm}\label{THM}
There is $c>0$ and a positive Lebesgue integrable 
function $\phi$ such that
\begin{equation*}
A_n = \int_0^c x^n\, \phi(x)\,dx
\end{equation*}
for $n=0,1,2,\cdots$.
\end{thm}

\begin{de}
We say $\phi$ is the \Def{moment density function} 
for $(A_n)$.
\end{de}

\subsection*{Notes}
I have tried to make the argument as
short as possible.  This means many asides and variations
have been removed.

Some of the proofs may be done using 
a computer algebra system (CAS).  I~used Maple 2015.
These are the sort of thing that---until
1980 or later---would have been done by paper-and-pencil 
computation.  I have added some of the
Maple as footnotes.

This result arose from a question asked by Alan Sokal.
It was posted on the MathOverflow discussion board \cite{MOmajer}.
Pietro Majer provided the idea to use the differential equation.

\begin{no}\label{notation:top}
We will use these values.
\begin{align*}
	c &= (\sqrt2+1)^4 = 17 + 12\sqrt{2} \approx 33.9705
	\\
	\co &= (\sqrt2+1)^{-4} = \frac{1}{c} = 
	34-c = 17-12\sqrt{2} \approx 0.0294
\end{align*}
\end{no}

\subsection*{The Differential Equation}
We proceed with a discussion of this\footnote{
\texttt{DE3:=x\carat2*(x\carat2-34*x+1)*%
diff(u(x),x\$3)+3*x*(2*x\carat2-51*x+1)*%
(diff(u(x),x\$2))\\
+(7*x\carat2-112*x+1)*diff(u(x),x)+(x-5)*u(x);}} 
third-order holonomic Fuchsian ODE:
\begin{align*}\label{DE3}
	x^2 (x^2-34x+1) u'''(x)
	+3x(2x^2-51x+1)u''(x)\qquad &
	\\ +
	(7x^2-112x+1)u'(x)+(x-5)u(x) = 0 .&
\tag{DE3}\end{align*}
We consider $x$ a complex variable, and sometimes
consider solutions
in the complex plane.

Differential equation (\ref{DE3}) has four singularities:
$\infty, 0, \co, c$.  They are all regular singular points.
Series solutions exist adjacent to each of them.
From the Frobenius ``series solution'' method\footnote{
\texttt{dsolve(DE3,u(x),series,x=c);}} 
\cite[Ch.~5]{BD}~\cite[Ch.~3]{BR}
we may describe these 
series solutions:
	
\begin{pr}\label{frobenius}
The general solution of {\rm(\ref{DE3})} near the
complex singular point $\infty$
has the form
\begin{align*}
	&A\left(\frac{1}{x} + o(x^{-1})\right)
	+
	B\left(\frac{\log x}{x} + o(x^{-1})\right)
	 +
	C\left(\frac{(\log x)^2}{x} + o(x^{-1})\right)
\end{align*}
as $x \to \infty$, for complex constants $A, B, C$.
The general solution of {\rm(\ref{DE3})} near the
singular point $0$
has the form
\begin{align*}
	A\left(1 + o(1)\right)
	+
	B\left(\log x + o(1)\right)
	+
	C\left((\log x)^2 + o(1)\right)
\end{align*}
as $x \to 0$, for complex constants $A, B, C$.
The general solution of {\rm(\ref{DE3})} near the
singular point $\co$
has the form
\begin{align*}
	&A\left(1-\frac{240+169\sqrt{2}}{48}(x-\co) +O(|x-\co|^2)\right)
	\\ & \qquad +
	B\left((x-\co)^{1/2}+O(|x-\co|^{3/2})\right)
	+
	C\left((x-\co)+O(|x-\co|^2) \right)
\end{align*}
as $x \to \co$, for complex constants $A, B, C$.
The general solution of {\rm(\ref{DE3})} near the
singular point $c$
has the form
\begin{align*}
	&A\left(1 - \frac{240-169\sqrt{2}}{48}(x-c) + O(|x-c|^2)\right)
	\\ &\qquad +
	B\left((x-c)^{1/2}+O(|x-c|^{3/2})\right)
	+
	C\left((x-c) + O(|x-c|^2) \right)
\end{align*}
as $x \to c$, for complex constants $A, B, C$.
\end{pr}

\begin{cor}
If $u(x)$ is any solution of {\rm(\ref{DE3})} on $(0,\co)$ or on
$(\co,c)$, then $u(x)$ has at worst logarithmic singularities.
So $u(x)$ is (absolutely, Legesgue) integrable.
\end{cor}

\begin{no}\label{basicnotation}
Four particular solutions  of (\ref{DE3}) will be named
for use here:

\noindent$\bullet$ Solution 
$u_\infty(x) = {1}/{x} + o(x^{-1})$ as $x \to \infty$,
defined in the complex plane
cut on the real axis interval $[0,c]$.

\noindent $\bullet$ Solution 
$u_0(x) = 1 + o(1)$ as $x \to 0^+$,
defined for $0 < x < \co$.

\noindent$\bullet$ Solution 
$v_0(x) = \log x + o(1)$ as $x \to 0^+$,
defined for $0 < x < \co$.

\noindent$\bullet$ Solution 
$v_2(x) = (c-x)^{1/2}+O(|x-c|^{3/2})$ as $x \to c^-$,
defined for $\co < x < c$.
\end{no}

\begin{figure}[htbp] 
   \centering
   \includegraphics[width=4.8in]{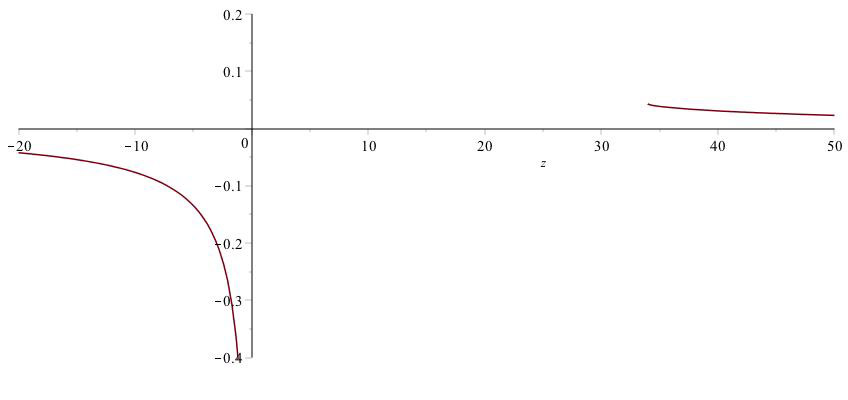} 
   \vskip-0.3in
   \caption{$u_\infty(x)$}
   \label{fig:uinfty}

   \centering
   \includegraphics[width=2.3in]{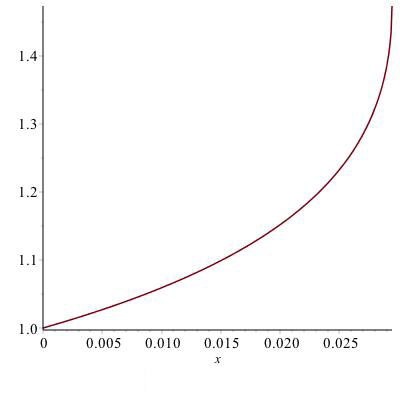} \hfil
   \includegraphics[width=2.3in]{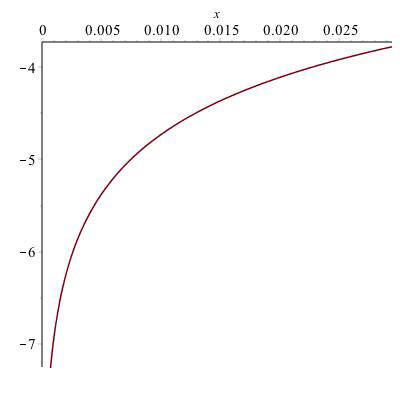} 
   \vskip-0.3in
   \caption{$u_0(x)$ and $v_0(x)$}
   \label{fig:u0v0}

   \centering
   \includegraphics[width=2.3in]{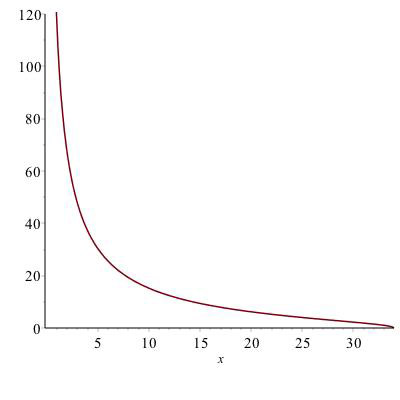} 
   \vskip-0.3in
   \caption{$v_2(x)$}
   \label{fig:v2}
\end{figure}

\begin{pr}\label{AperyGen}
The Maclaurin series for $u_0(x)$ is the generating function
for the \Apery\ sequence:
$$
	u_0(x) = \sum_{n=0}^\infty A_n x^n, 
	\qquad |x| < \co.
$$
\end{pr}
\begin{proof}
This may be checked by your CAS.  The recurrence
(\ref{rec})\footnote{
\texttt{Rec:=(n+1)\carat3*Q(n+1)-(34*n\carat3+%
51*n\carat2+27*n+5)*Q(n)+n\carat3*Q(n-1);}}
converted to a differential equation\footnote{
\texttt{gfun[rectodiffeq](\{Rec,Q(0)=1,Q(1)=5\},Q(n),u(x));}}
yields (\ref{DE3}).  Of course the radius of convergence
extends to the nearest singularity at $\co$.
\end{proof}

\begin{cor}
$u_0(x)>0$ for $0 < x < \co$.
\end{cor}

Determining the signs of $v_0$ and $v_2$ will be
more difficult.

\begin{pr}
The Laurent coefficients for $u_\infty(z)$ are
the \Apery\ numbers:
$$
	u_\infty(z) = \sum_{n=0}^\infty \frac{A_n}{z^{n+1}},
	\qquad |z| > c.
$$
\end{pr}
\begin{proof}
Check that if $u(x)$ is a solution
of (\ref{DE3}), then $w(z) =  u(1/z)/z$ is also a solution
of (\ref{DE3}).  Matching the boundary conditions, we
get
$$
	u_\infty(z) = \frac{1}{z} u_0\left(\frac{1}{z} \right) .
$$
Apply Prop.~\ref{AperyGen}.
\end{proof}

Note: In general, for other similar sequences that can be handled in
this same way: 

(a)~the generating function for the sequence, and

(b)~the moment density function
for the sequence

\noindent
satisfy \emph{different}
differential equations.

\subsection*{The Function  $\phi$}
Series solution $u_\infty$ of
(\ref{DE3}) is meromorphic  and single-valued
near $\infty$.   It continues
analytically to the complex plane with a cut on the
interval $[0,c]$ of the real axis.  We will still
use the notation
$u_\infty$ for that continuation.
Since the Laurent coefficients are all real, we have
\begin{equation}\label{symm}
	u_\infty\big(\,\overline{z}\,\big) = 
	\overline{u_\infty(z)}
\end{equation}
near $\infty$, and therefore on the whole domain.
In particular,
$u_\infty(z)$ is real for $z$ on the real axis (except the
cut, of course).  Define upper and lower values
on the cut $0 < x < c$:
$$
	u_\infty(x+i0) = \lim_{\delta \to 0+}
	u_\infty(x+i\delta),\quad
	u_\infty(x-i0) = \lim_{\delta \to 0+}
	u_\infty(x-i\delta).
$$
Then from (\ref{symm}) we have
\begin{equation}\label{cutsymm}
	u_\infty(x-i0) = \overline{u_\infty(x+i0)},
	\qquad 0<x<c.
\end{equation}

\begin{no}
$$
	\phi(x) = 
	\frac{1}{2\pi i}\big(u_\infty(x-i0)-u_\infty(x+i0)\big) .
$$
\end{no}
Function $u_\infty$ in the upper half plane
extends analytically to a solution
in a neighborhood of $(0,\co)$, and similarly
$u_\infty$ in the lower half plane.
Thus $\phi(x)$ restricted to $(0,\co)$ is a
solution of (\ref{DE3}),
since it is a linear combination of solutions.
In the same way, $\phi(x)$ restricted to $(\co,c)$ is a
solution of (\ref{DE3}).

See Figure~\ref{fig:Uall}; an enlargement shows the behavior
near the singular point $\co$.  We will see that $\phi$
has \sr\ asymptotics near the right endpoint $c$
(Prop.~\ref{alphav2}) and
logarithmic asymptotics near the left endpoint $0$
(Prop.~\ref{alphav0}).

\begin{figure}[htbp] 
   \centering
   \includegraphics[width=2.3in]{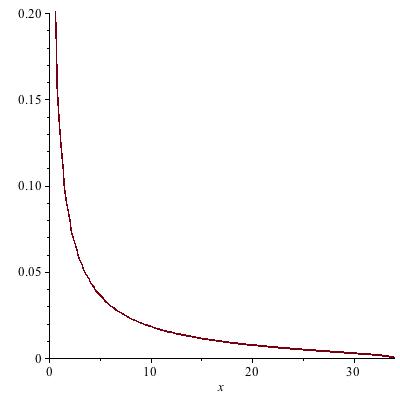} \hfil
   \includegraphics[width=2.3in]{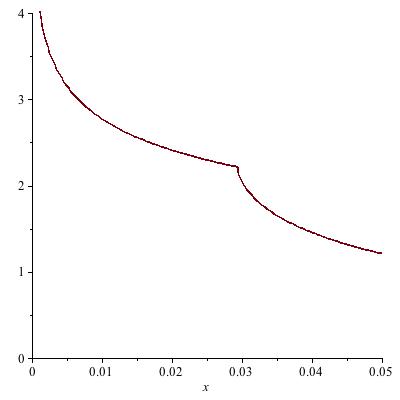} 
   \vskip-0.2in
   \caption{Moment density function $\phi(x)$}
   \label{fig:Uall}
\end{figure}

\begin{pr} The \Apery\ numbers satisfy
$A_k = \int_0^c x^k\,\phi(x)\,dx$,
$k = 0, 1, 2, \cdots$.
\end{pr}
\noindent\emph{Proof.}
Fix a nonnegative integer $k$.  For $\delta>0$, let
$\Gamma_\delta$ be the contour in the complex plane at distance
$\delta$ from $[0,c]$, as in Figure \ref{fig:contour}.
(Two line segments and two semicircles; traced
counterclockwise.)
Now $u_\infty$ has at worst logarithmic singularities, so
we have this limit:
$$
	\lim_{\delta\to 0+} \oint_{\Gamma_\delta}
	z^k u_\infty(z)\, dz
	=\int_0^c x^k \big(u_\infty(x-i0)-u_\infty(x+i0)\big) dx .
$$
On the other hand, $z^k u_\infty(z)$ is analytic on and outside
the contour $\Gamma_\delta$, except at $\infty$ where it has
an isolated singularity with residue
$A_k$.  Therefore
$$
	\oint_{\Gamma_\delta} z^k u_\infty(z)\, dz 
	=2\pi i A_k .
$$
Thus
$$
	A_k = \int_0^c
	 \frac{x^k}{2\pi i}
	\big(u_\infty(x-i0)-u_\infty(x+i0)\big)\,dx .\qquad\qed
$$
\begin{figure}[htbp] 
   \centering
   \includegraphics[width=3in]{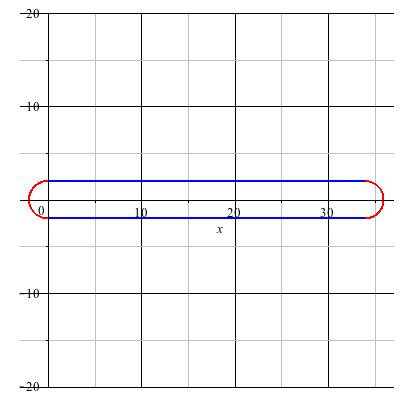} 
   \vskip-0.2in
   \caption{Contour $\Gamma_\delta$}
   \label{fig:contour}
\end{figure}

\emph{What remains to be proved}: $\phi$ is nonnegative
on $(0,c)$ (Cor.~\ref{phipos}).  From (\ref{cutsymm}) we know
that $\phi(x)$ is real on $(0,c)$.

\subsection*{Heun General Functions}
Some of the basic solutions in Notation~\ref{basicnotation}
may be represented in terms of Heun functions.  The Heun functions are described in
\cite{heun, DLMFheun, parametric}.  

\begin{de}
Let complex parameters $a,q,\alpha, \beta, \gamma, 
\delta, \epsilon$ be given satisfying
$a \ne 0$,
$\alpha+\beta+1=\gamma+\delta+\epsilon$,
$\delta \ne 0$, and
$\gamma \ne 0, -1, -2, \cdots$.
Define\footnote{
\texttt{HeunG(a,q,alpha,beta,gamma,delta,z)}}
the \Def{Heun general function}
\begin{equation}\label{heunseries}
	\Hn{a}{q}{\alpha}{\beta}{\gamma}{\delta}{z} 
	= \sum_{n=0}^\infty p_n z^n,
\end{equation}
where the Maclaurin coefficients satisfy initial conditions
\begin{equation*}
	p_0=1, \qquad p_1 = \frac{q}{a\gamma},
\end{equation*}
and recurrence
\begin{equation*}
	R_n p_{n+1} -(q+Q_n)p_n+P_n p_{n-1} = 0,
\end{equation*}
with
\begin{align*}
	R_n &=a(n+1)(n+\gamma) ,
	\\
	Q_n &= n\big((n-1+\gamma)(1+a)+a\delta+\epsilon\big) ,
	\\ 
	P_n &= (n-1+\alpha)(n-1+\beta) .
\end{align*}
\end{de}

This function satisfies the \Def{Heun general differential equation}
\begin{align*}
	w''(z)
	+\left(\frac{\gamma}{z}+\frac{\delta}{z-1}
	+\frac{\epsilon}{z-a}\right)w'(z)
	+\frac{\alpha\beta z - q}{z(z-1)(z-a)}\;w(z) = 0 .
\end{align*}
(Consult the references, or use your CAS to go from
the recurrence to the differential equation.)
This DE has singularities at $\infty, 0, 1, a$;
all regular singular points.  Convergence
of the series extends to the nearest singularity, so
the radius of convergence in (\ref{heunseries}) is
$\min\{1,|a|\}$.

\begin{pr} \label{basicheun}
Within the radius of convergence:
\begin{align*}
	u_0(x) &= \Hn{a_2}{q_4}{
	1/2}{1/2}{1}{1/2}{cx}^2
	\\
	v_2(x) &= \frac{(x-\co)(c-x)^{1/2}}{c-\co}
	\;\Hn{a_1}{q_1}{
	3/2}{3/2}{3/2}{1}{1-\co x} \cdot
	\\ &\qquad\qquad\qquad \cdot
	\;\Hn{a_1}{q_2}{
	1}{ 1}{ 1/2}{ 1}{ 1-\co x}
	\\
	u_\infty(z) &= \frac{1}{z} 
	\Hn{a_2}{q_4}{1/2}{1/2}{1}{1/2}{\frac{c}{z}}^2 ,
\end{align*}
where
\begin{align*}
	a_1 &= 1-\co^2 = -576+408\sqrt{2} \approx 0.9991
	\\
	a_2 &= c^2 =577+408\sqrt{2} \approx 1153.9991
	\\
	q_1 &= - \frac{1317}{4}+234\sqrt{2} \approx 1.676
	\\
	q_2 &= (\sqrt2+1)^{-1}(1+\co)
	= - 42+30\sqrt{2} \approx 0.4264
	\\
	q_4 &= \frac{5c}{2} = \frac{85}{2}+30\sqrt{2} \approx 84.93
\end{align*}
\end{pr}
\begin{proof}In each case verify
that it satisfies the differential equation\footnote{
\texttt{subs(u(x)=v2,DE3): simplify(\%);}} and
boundary properties\footnote{
\texttt{MultiSeries[series](v2,x=c,2);}} that specify the solution.
\end{proof}

\subsection*{All Coefficients Positive}
In some cases we can determine that all Maclaurin
coefficients of
a Heun general function
$$
	\Hn{a}{q}{\alpha}{\beta}{\gamma}{\delta}{z} 
$$
are positive.  When that is true, then
in particular this function will be positive and increasing
and convex on $(0,R)$ where $R = \min\{1,|a|\}$ is the radius
of convergence.

\begin{lem}\label{L2}
All Maclaurin coefficients are positive in
$$
	\Hn{a_1}{q_1}{3/2}{3/2}{3/2}{1}{z} ,
$$
where $a_1 = -576+408\sqrt{2}$ and
$q_1 = -\frac{1317}{4}+234\sqrt{2}$.
\end{lem}
\begin{proof}
Let $p_n$ be the Maclaurin coefficients.  Then
$$
	R_n p_{n+1} - (q_1+Q_n) p_n + P_n p_{n-1} = 0,
$$
with
\begin{align*}
	R_n &= \textstyle a_1 (n+1)(n+\frac{3}{2})
	\\
	Q_n &= \textstyle
	n \big((n+\frac{1}{2})(1+a_1) + a_1 + \frac{3}{2}\big)
	\\
	P_n &= \textstyle(n+\frac{1}{2})^2.
\end{align*}
Write $r_n = p_{n}/p_{n-1}$ and rearrange:
$$
	r_{n+1} = \frac{q_1+Q_n}{R_n} - 
	\frac{P_n}{R_n}\;\frac{1}{r_{n}} .
$$
Recall that $|a_1| < 1$; we expect $r_n \to 1/a_1$.
We claim: if
$$
	n \ge 45\qquad\text{and}\qquad
	1-\frac{1}{10n} < r_n < \frac{1}{a_1},
$$
then also
$$
	1-\frac{1}{10(n+1)} < r_{n+1} < \frac{1}{a_1} .
$$
Once the claim is proved, all that remains
is checking
that $p_0,\cdots, p_{45}$ are positive, and
$$
	1 - \frac{1}{450} < r_{45} < \frac{1}{a_1} .
$$
By induction we conclude that $r_n > 0$ for all $n \ge 45$.
So $p_n$ with $n > 45$ is a product of
positive numbers
$$
	p_{45} \;r_{46} r_{47} r_{48} \cdots r_n ,
$$
so $p_n > 0$.

\emph{Proof of the claim.}  Since
$$
	r \mapsto \frac{q_1+Q_n}{R_n} - 
	\frac{P_n}{R_n}\;\frac{1}{r}
$$
is an increasing function, we need only check
$$
	1-\frac{1}{10(n+1)} < 
	\frac{q_1+Q_n}{R_n} - 
	\frac{P_n}{R_n}\;a_1
	 < \frac{1}{a_1}
$$
and
$$
	1-\frac{1}{10(n+1)} < 
	\frac{q_1+Q_n}{R_n} - 
	\frac{P_n}{R_n}\;\frac{1}{1-\frac{1}{10 n}}
	 < \frac{1}{a_1}
$$
where $n \ge 45$.
Your CAS can be used for this.
\end{proof}

A warning for the computations.  If you do this using $20$-digit
arithmetic---as I did at first---you may erroneously conclude that it is
false.  You may see negative coefficients.
With exact arithmetic, we find that $r_{45}$ involves 
integers with more than $100$ digits.  To compare 
$\sqrt{2}$ to a rational
number with $100$-digit numerator and denominator, there
are two methods:
we can square those $100$-digit numbers, or we can use a 
decimal value
of $\sqrt{2}$ accurate to more than $100$ places.
Of course a modern CAS can do either.

\begin{lem}\label{L6}
All Maclaurin coefficients are positive in
$$
	\Hn{a_1}{q_2}{1}{1}{1/2}{1}{z} ,
$$
where $a_1= -576 + 408\sqrt{2}$ and $q_2= -42+30\sqrt{2}$.
\end{lem}
\begin{proof}
The proof is similar to Lemma~\ref{L2}.
Let $p_n$ be the coefficients, and
$r_n = p_n/p_{n-1}$.  Then
$$
	r_{n+1} = \frac{q_2+Q_n}{R_n}
	-\frac{P_n}{R_n}\;\frac{1}{r_n},
$$
with
\begin{align*}
	R_n &= \textstyle a_1(n+1)(n+\frac{1}{2}) ,
	\\
	Q_n &= \textstyle 
	n\big((n-\frac{1}{2})(1+a_1)+a_1 + \frac{3}{2}\big) ,
	\\
	P_n &= \textstyle n^2 .
\end{align*}
We claim: If
$$
	n \ge 18\qquad\text{and}\qquad
	1-\frac{1}{4n} < r_n < \frac{1}{a_1} ,
$$
then also
$$
	1-\frac{1}{4(n+1)} < r_{n+1} < \frac{1}{a_1} .
$$
The remainder of the proof is similar to Lemma~\ref{L2}.
\end{proof}

\begin{pr}\label{v2pos}
$v_2(x)>0$ for $\co < x < c$.
\end{pr}
\begin{proof}
By Lemma~\ref{L2}, all Maclaurin coefficients of
$$
	\Hn{a_1}{q_1}{
	3/2}{3/2}{3/2}{1}{z}
$$
are positive.  It has radius of convergence
$a_1 = 1-\co^2$, so
$$
	\Hn{a_1}{q_1}{
	3/2}{3/2}{3/2}{1}{1-\co x} > 0
$$
for all $x$ with $\co < x < c$.  By Lemma~\ref{L6},
all Maclaurin  coefficients of
$$
	\Hn{a_1}{q_2}{1}{ 1}{1/2}{ 1}{z}
$$
are positive.  Again,
$$
	\Hn{a_1}{q_2}{1}{ 1}{1/2}{ 1}{1-\co x} > 0
$$
for all $x$ with $\co < x < c$.  Also
$$
	\frac{(x-\co)(c-x)^{1/2}}{c-\co}
$$
is positive on $(\co,c)$.  The product of
three positive factors
is $v_2(x)$ on $(\co,c)$, so $v_2(x)>0$.
\end{proof}

\subsection*{Hypergeometric Function}
Some Heun functions can be expressed in terms of
hypergeometric ${}_2F_1$ functions \cite[Chap.~2--3]{AAR}.  Here,
we will use only one of them.\footnote{%
\texttt{hypergeom([1/3,2/3],[1],z)}}

\begin{de}
$\displaystyle
{}_2F_1\left(\frac{1}{3},\frac{2}{3};1;z\right)
=\sum_{n=0}^\infty \frac{(3n)!}{(n!)^3}\,\frac{z^n}{27^n}$
\end{de}

\begin{lem}\label{gauss}
{\rm(a)}~${}_2F_1\left(\frac{1}{3},\frac{2}{3};1;z\right)$
has radius of convergence $1$.
{\rm(b)}~For $0<z<1$, we have
$\;{}_2F_1\left(\frac{1}{3},\frac{2}{3};1;z\right) > 1$.
{\rm(c)}~As $\delta \to 0^+$,
\begin{align*}
	{}_2F_1\left(\frac{1}{3},\frac{2}{3};1;1-\delta\right) &=
	-\frac{\sqrt{3}}{2\pi}\log \delta
	+\frac{3\sqrt{3}\log 3}{2\pi} + o(1) .
\end{align*}
\end{lem}
\begin{proof}
(a) Ratio test.

(b) All Maclaurin coefficients are positive, and the
constant term is $1$.

(c) Due to Gauss (or perhaps Goursat?), see
\cite[Thm.~2.1.3]{luke}~\cite{wolfram},
\begin{align*}
	{}_2F_1\left(\frac{1}{3},\frac{2}{3};1;1-\delta\right) &=
	\frac{\Gamma(1)}{\Gamma(\frac{1}{3})\Gamma(\frac{2}{3})}
	\Big[\textstyle\log\frac{1}{\delta} - 2\gamma 
	-\psi\left(\frac{1}{3}\right)-\psi\left(\frac{2}{3}\right)
	\Big] + o(1)
	\\ &=
	-\frac{\sqrt{3}}{2\pi}\log \delta
	+\frac{3\sqrt{3}\log 3}{2\pi} + o(1) .
\end{align*}
Here $\gamma$ is Euler's constant and $\psi$ is the
digamma function.  Use \cite[Thm.~1.2.7]{AAR} to
evaluate the digamma of a rational number.
\end{proof}

\begin{lem}\label{Fderiv}
Let the degree $1$ Taylor polynomial for
${}_2F_1\left(\frac{1}{3},\frac{2}{3};1;z\right)$ at
$z_0=1/(2^{3/2}(\sqrt2+1))$ be
${}_2F_1\left(\frac{1}{3},\frac{2}{3};1;z\right)
= S_0 + S_1 \cdot(z-z_0) + o(|z-z_0|)$ as $z \to z_0$.
Then
$$
	S_0\cdot(3S_1+\sqrt{2}S_0) = \frac{3^{3/2}2^{1/2}}{\pi} .
$$
\end{lem}
\begin{proof}
See Appendix.
\end{proof}

To complete the proof of Theorem 1, we do not need the
exact value in Lemma~\ref{Fderiv}, but only that it is positive;
which is clear from the fact that all
Maclaurin coefficients of ${}_2F_1\left(\frac{1}{3},\frac{2}{3};1;z\right)$
are positive and $z_0$ is positive.

\begin{no}
\begin{align*}
	\mu(x) &= \frac{\left(3-3x-\sqrt{x^2-34x+1}\;
	\right)^{1/2}}{\sqrt{2}\;(x+1)}
	\\
	\mu_2(x) &= \frac{\left(3-3x+\sqrt{x^2-34x+1}\;
	\right)^{1/2}}{\sqrt{2}\;(x+1)}
	\\
	\lambda(x) &=
	\frac{x^3+30 x^2-24 x + 1 
	-(x^2-7x+1)\sqrt{x^2-34x+1}}{2(x+1)^3}
	\\
	\lambda_2(x) &=
	\frac{x^3+30 x^2-24 x + 1 
	+(x^2-7x+1)\sqrt{x^2-34x+1}}{2(x+1)^3}
\end{align*}
 (See Figures \ref{fig:mu} and~\ref{fig:lambda}.)
 \end{no}

\begin{figure}[thbp] 
   \centering
   \includegraphics[width=3in]{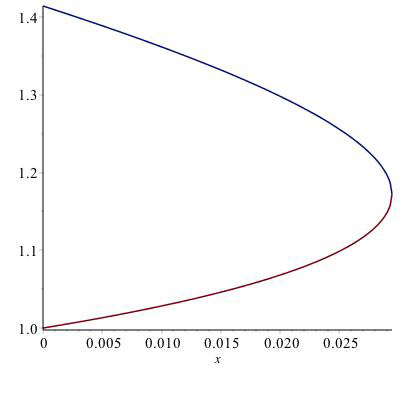} 
   \vskip-0.4in
   \caption{$\mu$ (bottom) and $\mu_2$ (top)}
   \label{fig:mu}
\end{figure}

\begin{figure}[htbp] 
   \centering
   \includegraphics[width=3in]{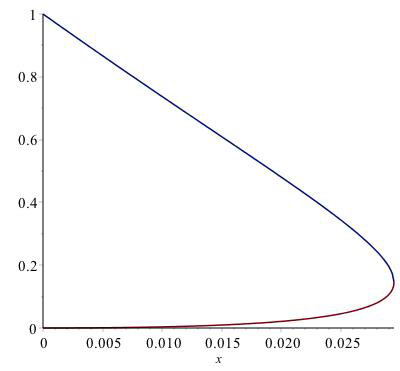} 
   \vskip-0.2in
   \caption{$\lambda$ (bottom) and $\lambda_2$ (top)}
   \label{fig:lambda}
\end{figure}

\begin{lem}\label{mulambda}
For $0<x<\co$, we have $\mu(x) >1$,
$\mu_2(x)>1$, $0 < \lambda(x)< 1$, and $0 < \lambda_2(x) < 1$.
\end{lem}
\begin{proof}
Elementary inequalities.
\end{proof}

\begin{lem}\label{fghyper}
As $x \to 0^+$,
\begin{align*}
	\mu(x)
	\;{}_2F_1\left(\frac{1}{3},\frac{2}{3};1;\lambda(x)\right) 
	&= 1 + \frac{5}{2} \,x + O(x^2),
	\\
	\mu_2(x)
	\;{}_2F_1\left(\frac{1}{3},\frac{2}{3};1;\lambda_2(x)\right) 
	&=\frac{-\sqrt{3}}{\pi\sqrt{2}}\,\log x + o(1).
\end{align*}
The second one indeed has constant term zero.
\end{lem}
\noindent\emph{Proof.}
Compute (as $z \to 0$ and $x \to 0$):
\begin{align*}
	\mu(x) &= 1 + \frac{5}{2} x + O(x^2)
	\\
	\lambda(x) &= 27x^2 + O(x^3)
	\\
	{}_2F_1\left(\frac{1}{3},\frac{2}{3};1;z\right) &=
	1 + \frac{2}{9} z + O(z^2)
	\\
	{}_2F_1\left(\frac{1}{3},\frac{2}{3};1;\lambda(x)\right) &=
	1 + 6 x^2 + O(x^3)
	\\
	\mu(x)\;{}_2F_1\left(\frac{1}{3},\frac{2}{3};1;
	\lambda(x)\right) &= 1 + \frac{5}{2} x + O(x^2)
\end{align*}
For the second one, we apply Lemma~\ref{gauss}(c).
As $x \to 0$:
\begin{align*}
	\mu_2(x) &= \sqrt{2} - \frac{7}{\sqrt{2}} x + O(x^2)
	\\
	\lambda_2(x) &= 1 - 27 x + O(x^2)
	\\
	{}_2F_1\left(\frac{1}{3},\frac{2}{3};1;
	\lambda_2(x)\right) &=
	 -\frac{\sqrt{3}}{2\pi} \log (27 x)
	+\frac{3\sqrt{3}\log 3}{2\pi} + o(1)
	\\ &
	= -\frac{\sqrt{3}}{2\pi}(\log 27 + \log x)
	+\frac{3\sqrt{3}\log 3}{2\pi} + o(1)
	\\ &
	= -\frac{\sqrt{3}}{2\pi}\log x
	 + o(1)
	\\
	\mu_2(x){}_2F_1\left(\frac{1}{3},\frac{2}{3};1;
	\lambda_2(x)\right) &= -\frac{\sqrt{3}}{\sqrt{2}\,\pi}\log x + o(1)
	\qquad\qed
\end{align*}

\begin{pr}\label{uhyper}
\begin{align*}
	u_0(x) &= \mu(x)^2
	\;{}_2F_1\left(\frac{1}{3},\frac{2}{3};1;\lambda(x)\right)^2 ,
	\\
	v_0(x) &= - \frac{2\pi}{\sqrt{3}(x+1)}
	\;{}_2F_1\left(\frac{1}{3},\frac{2}{3};1;\lambda(x)\right)
	\;{}_2F_1\left(\frac{1}{3},\frac{2}{3};1;\lambda_2(x)\right) ,
	\\
	u_\infty(z) &=
	\frac{1}{z}
	\mu\left(\frac{1}{z}\right)^2
	{}_2F_1\left(\frac{1}{3},\frac{2}{3};1;
	\lambda\left(\frac{1}{z}\right)\right)^2 .
\end{align*}
\end{pr}
\begin{proof}
Note: $\mu(x)\mu_2(x) = \sqrt{2}/(x+1)$.
Verify that these expressions
satisfy (\ref{DE3}) as usual.  Then
verify the asymptotics using Lemma~\ref{fghyper}.
\end{proof}

How were these formulas found?  The first one is from
Mark van Hoeij
\cite[\texttt{A005259}]{OEIS};
I do not know how he found it.  But then
it is natural to try the other
\sr, since that will still satisfy the same
differential equation.

\begin{pr}\label{v0neg}
$v_0(x) < 0$ for $0 < x < \co$.
\end{pr}
\begin{proof}  For $0 < x < \co$:
By Lemma~\ref{mulambda}, $0<\lambda(x)<1$, so by
Lemma~\ref{gauss}(b),
$\;{}_2F_1\left(\frac{1}{3},\frac{2}{3};1;\lambda(x)\right) > 0$.
Similarly,
$\;{}_2F_1\left(\frac{1}{3},\frac{2}{3};1;\lambda_2(x)\right)>0$.
\end{proof}

\subsection*{The Two Endpoints}

\begin{pr} \label{alphav2}
On interval $(\co,c)$ we have exactly
$$
\phi(x) = \frac{v_2(x)}{2^{5/4}(\sqrt2+1)^4\pi^2}.
$$
\end{pr}
\begin{proof}
We examine the solution $u_\infty(x)$  of (\ref{DE3}) on
the interval $(c,+\infty)$.
As $\delta \to 0^+$, the Frobenius series solution shows that
\begin{equation}\label{uinfc}
	u_\infty(c+\delta) = A + B\sqrt{\delta} + C\delta +O(\delta^{3/2}) 
\end{equation}
for some real constants $A,B,C$;
we will have to evaluate the constant $B$ below.  Following (\ref{uinfc})
around the point $c$ by a half-turn in either direction, we get
\begin{align*}
	u_\infty(c-\delta-i0) &=
	A + B (-i) \sqrt{\delta} - C\delta +O(\delta^{3/2})
	\\
	u_\infty(c-\delta+i0) &=
	A + B i \sqrt{\delta} - C\delta +O(\delta^{3/2})
	\\
	\phi(c-\delta) &=
	\frac{1}{2\pi i}\big(u_\infty(c-\delta-i0)
	 - u_\infty(c-\delta+i0)\big)
	\\ & = \frac{0 A - 2 B i\sqrt{\delta} + 0 C\delta}{2 \pi i}
	+O(\delta^{3/2})
	\\ &=
	\frac{-B}{\pi}\sqrt{\delta} + O(\delta^{3/2}).
\end{align*}
Therefore $\phi(x) = (-B/\pi) v_2(x)$ on $(\co,c)$.

On interval $(c,+\infty)$, we have
$$
	u_\infty(x) =
	\frac{1}{x}
	\mu\left(\frac{1}{x}\right)^2
	{}_2F_1\left(\frac{1}{3},\frac{2}{3};1;
	\lambda\left(\frac{1}{x}\right)\right)^2 .
$$
The value $\lambda(1/x)$ stays inside the unit disk, so
no analytic continuation is required.  
Now $\lambda(\co) = \lambda_2(\co) = 1/(2^{3/2}(\sqrt2+1))$,
called $z_0$ in Lemma~\ref{Fderiv}.  Let $S_0, S_1$ also
be as in Lemma~\ref{Fderiv}.
As $\delta \to 0^+$,
\begin{align*}
	\frac{1}{c+\delta} &= \frac{1}{(\sqrt2+1)^4} +O(\delta)
	\\
	\mu\left(\frac{1}{c+\delta}\right) &= 
	\frac{(\sqrt2+1)}{2^{1/4}3^{1/2}} 
	- \frac{1}{4\cdot3\cdot(\sqrt2+1)} \sqrt{\delta} +O(\delta)
	\\
	\lambda\left(\frac{1}{c+\delta}\right) &=
	\frac{1}{2^{3/2}(\sqrt2+1)} 
	- \frac{\sqrt{3}}{2^{9/4}(\sqrt2+1)^2}\sqrt{\delta} +O(\delta)
	\\
	{}_2F_1\left(\frac{1}{3},\frac{2}{3};1;
	\lambda\left(\frac{1}{c+\delta}\right)\right) &=
	S_0 - \frac{\sqrt{3}}{2^{9/4}(\sqrt2+1)^2}S_1\sqrt{\delta} + O(\delta)
	\\
	u_\infty(c+\delta) &=
	\frac{S_0^2}{3\sqrt{2}(\sqrt2+1)^2}
	- \frac{S_0(3S_1+\sqrt{2}S_0)}{
	2^{7/4}3^{3/2}(\sqrt2+1)^4}\sqrt{\delta}+O(\delta) 
	\\ &=
	\frac{S_0^2}{3\sqrt{2}(\sqrt2+1)^2}
	- \frac{1}{2^{5/4}(\sqrt2+1)^4\pi}\sqrt{\delta}+O(\delta) .
\end{align*}
So we get $B=-1/(2^{5/4}(\sqrt2+1)^4\pi)$.
\end{proof}

\begin{pr}\label{alphav0}
On interval $(0,\co)$ we have exactly
$\phi(x) = -6v_0(x)/\pi^2$.
\end{pr}
\begin{proof}
We examine the solution $u_\infty(x)$  of (\ref{DE3}) on
the interval $(-\infty,0)$.
As $\delta \to 0^+$, the Frobenius series solution shows that
\begin{equation}\label{uinf0}
	u_\infty(-\delta) = A + B\log\delta + C(\log \delta)^2
	+ o(1)
\end{equation}
for some real constants $A,B,C$;
we will have to evaluate the constants $B$ and
 $C$ below.  Following (\ref{uinf0})
around the point $0$ by a half-turn in either direction, we get
\begin{align*}
	u_\infty(\delta-i0) &=
	A + B (\log\delta + i\pi) + 
	C(\log\delta+ i\pi)^2 +o(1)
	\\
	u_\infty(\delta+i0) &=
	A + B (\log\delta - i\pi) + 
	C(\log\delta- i\pi)^2 +o(1)
	\\
	\phi(\delta) &=
	\frac{1}{2\pi i}\big(u_\infty(\delta-i0)
	 - u_\infty(\delta+i0)\big)
	\\ & =
	\frac{2 B i \pi + 4 C i \pi \log \delta}{2\pi i} + o(1)
	\\ &=
	B + 2 C \log \delta + o(1)
\end{align*}
Therefore $\phi(x) = B u_0(x) + 2 C v_0(x)$
on $(0,\co)$.

On interval $(-\infty, 0)$, we have
$$
	u_\infty(x) =
	\frac{1}{x}
	\mu\left(\frac{1}{x}\right)^2
	{}_2F_1\left(\frac{1}{3},\frac{2}{3};1;
	\lambda\left(\frac{1}{x}\right)\right)^2 .
$$
Argument $\lambda(1/x)$ stays inside the unit disk, so
this is an easy
analytic continuation of $u_\infty$.  As $\delta \to 0^+$,
\begin{align*}
	\frac{1}{-\delta} &= \frac{-1}{\delta} +O(1)
	\\
	\mu\left(\frac{1}{-\delta}\right)^2 &= 
	\delta+O(\delta^2)
	\\
	\lambda\left(\frac{1}{-\delta}\right) &=
	1-27\delta^2+O(\delta^3)
\end{align*}
So by Lemma~\ref{gauss}(c),
\begin{align*}
	{}_2F_1\left(\frac{1}{3},\frac{2}{3};1;
	\lambda\left(\frac{1}{-\delta}\right)\right) &=
	\frac{-\sqrt{3}}{2\pi}\;2\log\delta + o(1)
	\\
	u_\infty(-\delta) &= \frac{-3}{\pi^2}(\log \delta)^2 + o(1) .
\end{align*}
Thus we get $B=0$ and $C=-3/\pi^2$.
\end{proof}

\begin{cor}\label{phipos}
The moment density $\phi$ may be written
$$
	\phi(x) = \begin{cases}
	\displaystyle \;\frac{-6}{\pi^2} \,v_0(x), & 0 < x < \co,
	\\{}\\
	\displaystyle \;
	\frac{1}{2^{5/4}(\sqrt2+1)^4\pi^2} \,v_2(x) , &\co \le x \le c .
	\end{cases}
$$
It is positive on $(0,\co) \cup (\co,c)$.
\end{cor}
\begin{proof}
For $0 < x < \co$, by Prop.~\ref{v0neg}
$v_0(x) < 0$.  For $\co < x < c$, by Prop.~\ref{v2pos}
$v_2(x) > 0$.
\end{proof}

This completes the proof of Theorem~\ref{THM}.

\subsection*{Appendix: Lemma 19}
\emph{Let the degree $1$ Taylor polynomial for
${}_2F_1\left(\frac{1}{3},\frac{2}{3};1;z\right)$ at
$z_0=1/(2^{3/2}(1+\sqrt2\;))$ be
${}_2F_1\left(\frac{1}{3},\frac{2}{3};1;z\right)
= S_0 + S_1 \cdot(z-z_0) + o(|z-z_0|)$ as $z \to z_0$.
Then}
$$
	S_0\cdot(3S_1+\sqrt{2}S_0) = \frac{3^{3/2}2^{1/2}}{\pi} .
$$
We will see that
a proof can be given using modular parameterizations.
See \cite{apostol, DS, LMFDB}.

\emph{Notation.} Let $\tau$ be a complex variable with 
$\operatorname{Im} \tau > 0$
and let $q = e^{2\pi i \tau}$ so that $0 < |q|<1$.
Thus
$$
	\frac{dq}{d \tau} = 2\pi i e^{2 \pi i \tau} = 2 \pi i q,
	\qquad
	\frac{d \tau}{dq} = \frac{1}{2\pi i q} .
$$
We will use a  derivation $\th$ defined by
$$
	\th = q\;\frac{d}{dq} = \frac{1}{2\pi i}\;\frac{d}{d\tau} .
$$
Sometimes this derivative will be written as
a superscript:
$$
	q\frac{d}{dq}f =
	\frac{1}{2\pi i}\;\frac{d}{d\tau} f =
	\th f =
	f^{\th}(\tau) .
$$

We will use the following functions:  The Dedekind eta function,
\begin{equation}\label{etaprod}
	\eta = q^{1/24}\prod_{n=1}^\infty (1-q^n) .
\end{equation}
One of the McKay-Thompson modular functions \cite[\texttt{A030182}]{OEIS}
\begin{align*}
	j_{3B}(\tau) &= \frac{\eta(\tau)^{12}}{\eta(3\tau)^{12}}
	\\ &
	=q^{-1} - 12 + 54q - 76q^2 - 243q^3 + 1188q^4 
	- 1384q^5
	\\ &\qquad\qquad - 2916q^6 + 11934q^7 - 11580q^8 - 21870q^9 + \dots
\end{align*}
And: the Eisenstein series $E_2$;
the logarithmic derivative of $\eta^{24}$, in the sense
$$
	E_2 = 24\;\frac{\eta^{\th}}{\,\eta\,} .
$$
The Fourier series is
 \cite[\texttt{A006352}]{OEIS}
$$
	E_2 = 1-24\sum_{k=1}^\infty \sigma(k) q^k ,
$$
where the number-theoretic function $\sigma(k)$ is the sum of the (positive integer) divisors
of $k$.  This is from logarithmic differentiation of $(\ref{etaprod})$.

Logarithmic differentiation of $j_{3B}$ yields
$$
	\frac{j^{\th}_{3B}(\tau)}{j_{3B}(\tau)}
	= \frac{1}{2}E_2(\tau) - \frac{3}{2}E_2(3\tau) .
$$

Our calculations will require values of
$\eta$ and $E_2$ for
$\tau=i\sqrt{6}$ and $i\sqrt{6}/3$.  These values
may be written in
terms of the constant
$$
	\FF := {}_2F_1\left(\frac16,\frac13;1;\frac12\right) 
	\approx 1.0354935.
$$
In fact,
$$
	\FF = \frac{3^{1/4}\sqrt{2}}{8\pi^{3/2}}
	\sqrt{\Gamma\left(\frac{1}{24}\right)
	\Gamma\left(\frac{5}{24}\right)
	\Gamma\left(\frac{7}{24}\right)\Gamma\left(\frac{11}{24}\right)} ,
$$
but we will not need this value for the proof of Lemma 19.

For evaluation at $\tau = i\sqrt{6}/3$,
we will use these values:
\begin{align*}
	\tau&= \frac{i\sqrt{6}}{3} \\ 
	3\tau &= i\sqrt{6}\\ 
	\eta(\tau) &= \frac{2^{3/4}3^{7/8}(1+\sqrt{2}\;)^{1/12}\FF^{1/2}}{6}\\ 
	\eta(3\tau) &= \frac{2^{3/4}3^{5/8}(\sqrt{2}-1)^{1/12}\FF^{1/2}}{6}\\ 
	E_2(\tau) &=\frac{3\sqrt{6}}{2\pi}+\frac{(\sqrt{2}+1)\FF^2}{\sqrt2}\\
	E_2(3\tau) &=\frac{\sqrt{6}}{2\pi}+\frac{(\sqrt{2}+1)\FF^2}{3\sqrt2}\\
	j_{3B}(\tau) &= 27(1+\sqrt{2}\;)^2\\
	\frac{27}{j_{3B}(\tau)+27} &= \frac{1}{2}-\frac{\sqrt{2}}{4} = z_0\\
	j_{3B}^{\th}(\tau) &= 27(1+\sqrt{2}\;)^2\FF^2 .
\end{align*}

Begin with this modular parameterization for
${}_2F_1(\frac13,\frac23;1;z)$:
\begin{equation}\label{1323param}
	{}_2F_1\left(\frac13,\frac23;1;\frac{27}{j_{3B}(\tau)+27}\right)
	=
	\frac{\eta(3\tau)^3}{\eta(\tau)}\;(j_{3B}(\tau)+27)^{1/3} .
\end{equation}
This is proved using a standard method:  Each side is
a modular form of weight $2$ and level $3$.
The space of such modular forms is finite-dimensional,
so it suffices to check that a certain number of
the Fourier coefficients agree.

Write $F(z) = {}_2F_1(\frac13,\frac23;1;z)$.  That means
$S_0 = F(z_0)$ and $S_1=F'(z_0)$.  
Substitution of $\tau = i\sqrt{6}/3$ into (\ref{1323param}) yields
\begin{equation}\label{S0val}
	S_0 = \FF .
\end{equation}

Logarithmic differentiation
of (\ref{1323param}) yields
\begin{align*}
	-\frac{27}{2}
	\frac{\eta(\tau)}{\eta(\tau)^3}
	j_{3B}^{\th}(\tau)
	(E_2(\tau)-3E_2(3\tau))
	\frac{j_{3B}(\tau)}{(j_{3B}(\tau)+27)^{1/3}}
	F'\left(\frac{27}{j_{3B}(\tau)+27}\right)&
	\\ =
	\frac{9E_2(3\tau)-E_2(\tau)}{24}
	+\frac{E_2(\tau)-3E_2(3\tau)}{6}
	\frac{j_{3B}(\tau)}{j_{3B}(\tau)+27}& .
\end{align*}
Substitution of $\tau = i\sqrt{6}/3$ yields
$$
	\frac{\FF}{8}F'(z_0) = \frac{\sqrt{6}}{8\pi}
	-\frac{\sqrt{2}\FF^2}{24} ,
$$
so
\begin{equation}\label{S1val}
	S_1 = \frac{\sqrt{6}}{\pi\FF}
	-\frac{\sqrt{2}\FF}{3} .
\end{equation}
Finally, from (\ref{S0val}) and (\ref{S1val})  we obtain
$$
	S_0\;(3S_1+\sqrt{2}\;S_0) =
	\frac{3\sqrt{6}}{\pi}
$$
as claimed.

\end{document}